\documentclass[11pt]{amsart}

\usepackage{mathrsfs}
\usepackage{amssymb}
\usepackage{amsthm}
\usepackage{dsfont}
\usepackage{amssymb}
\usepackage{amsmath,amscd}
\usepackage[all,cmtip]{xy}
\usepackage{mathrsfs}
\usepackage{multirow}
\usepackage{hyperref}
\usepackage{graphicx}
\usepackage{amsfonts, amsthm}
\usepackage{fullpage}
\usepackage{enumerate}
\usepackage{appendix}
\usepackage{cite}

\theoremstyle{plain}
\newtheorem{thm}[subsection]{Theorem}
\newtheorem{lemma}[subsection]{Lemma}
\newtheorem{prop}[subsection]{Proposition}

\theoremstyle{definition}
\newtheorem{rmk}[subsection]{Remark}
\newtheorem{defn}[subsection]{Definition}
\newtheorem{example}[subsection]{Example}

\let\a\alpha

\let\l\lambda
\let\m\mu

\let\p\psi
\let\o\omega
\let\r\rho
\let\s\sigma

\let\th\theta

\def\scr{\mathscr}
\def\cal{\mathcal}

\let\G\Gamma

\let\O\Omega

\newcommand{\ra}{\longrightarrow}

\newcommand{\stab}{\operatorname{stab}}
\newcommand{\im}{{\rm im}\:}

\newcommand{\DD}{{\mathbb D}}
\newcommand{\GG}{{\mathbb G}}
\newcommand{\F}{{\mathbb F}}
\newcommand{\Z}{{\mathbb Z}}
\newcommand{\C}{{\mathbb C}}
\newcommand{\Q}{{\mathbb Q}}
\newcommand{\R}{{\mathbb R}}

\newcommand{\cA}{\mathcal{A}}

\newcommand{\cC}{\mathcal{C}}

\newcommand{\cF}{\mathcal{F}}
\newcommand{\cG}{\mathcal{G}}
\newcommand{\cE}{\mathcal{E}}

\newcommand{\cL}{{\mathcal{L}}}
\newcommand{\cO}{\mathcal{O}}
\newcommand{\cT}{\mathcal{T}}
\newcommand{\cV}{\mathcal{V}}

\newcommand{\sL}{{\mathscr L}}

\newcommand{\Spec}{{\mbox{Spec }}}

\newcommand{\cris}{{\mathrm{Cris}}}
\newcommand{\cri}{{\mathrm{cris}}}
\newcommand{\Def}{{\mathrm{Def}}}
\newcommand{\Aut}{{\mathrm{Aut}}}
\newcommand{\et}{{\mathrm{et}}}
\newcommand{\mult}{{\mathrm{mult}}}

\newcommand{\Rep}{{\mathrm{Rep}}}

\newcommand{\iso}{{\mathrm{Isocris}}}

\newcommand{\Lie}{{\mbox{Lie }}}
\newcommand{\Pic}{{\text{Pic }}}

\newcommand{\univ}{{\mathrm{univ}}}
\newcommand{\Hom}{{\mathrm{Hom}}}

\newcommand{\Pp}{{(\s^*-\mathrm{Id})}}

\normalfont\upshape
\begin{document}

\title{Tensor decomposition of isocrystals characterizes Mumford curves}
\author{Jie Xia}
\maketitle

\begin{abstract}
We seek an appropriate definition for a Shimura curve of Hodge type in positive characteristics via characterizing curves in positive characteristics which are reduction of Shimura curve over $\C$. In this paper, we study the liftablity of a curve in the moduli space of principally polarized abelian varieties over $k, \text{char } k=p$. We show that in the generic ordinary case, some tensor decomposition of the isocrystal associated to the family imply that this curve can be lifted to a Shimura curve. 
\end{abstract}

\section{Introduction}
\subsection{Motivation and the main result}

Classical Shimura varieties are defined over number fields. As far as I know, there is no definition of Shimura varieties in positive characteristic. The theory of integral models of Shimura varieties has been developed to be the main tool to study Shimura varieties in positive characteristic.  In this paper, we seek an intrinsic definition of Shimura varieties in characteristic $p$ which guarantees that they are reductions of classical Shimura varieties. 

Such a definition is not a problem for Shimura varieties of PEL type. It has been stated clearly in many references, such as \cite{Kott} and \cite{Milne}. Roughly, Shimura varieties of PEL type have a natural moduli interpretation: polarized abelian varieties with certain endomorphism algebras. Since the endomorphism algebra of abelian varieties can be extended to mixed characteristic, we can state the moduli problem over a field of positive characteristic as well and take the solution to be the PEL type Shimura varieties in characteristic $p$. 

However, for Shimura varieties of Hodge type, we can not use the same strategy. Though it also has a natural moduli interpretation, without assuming Hodge conjecture, the Hodge cycles involved in the moduli problem may not be algebraic and they have no natural translation to positive characteristic. Via the work of Kisin \cite{Kisin}, we are able to define the integral canonical model of Shimura varieties of Hodge type and also a universal family over the integral model. So a natural candidate for the definition should be the reduction of an integral canonical model at a prime $\frak{p}$. The disadvantage is that this definition is not intrinsic in characteristic $p$. Our goal is to find a characterization in positive characteristic of such a reduction which can serve as an intrinsic definition. We mainly focus on a remarkable example constructed by Mumford.

Mumford gives a countably many moduli functors of abelian fourfolds in \cite{Mum}, whose Shimura datum is given via the corestriction of an quaternion algebra over a cubic totally real field. In \cite{Zuo}, Viehweg and Zuo generalize the construction. One can consider a quaternion division algebra $D$ defined over any totally real number field $F$ of degree $m+1$ over $\Q$, which is ramified at all infinite places except one and $\text{Cor}_{F/\Q}(D)=M_{2^{m+1}}(\Q)$. Then we obtain the following types of moduli functors of abelian varieties with special Mumford-Tate group 
\[Q=\{x\in D^*| x\bar x =1\}.\]
With suitable level structures, the resultant Shimura curves $M$ is proper and smooth which we call them Mumford curves. Since the level structure is not fixed, Mumford curves stand for a class of possible base curves, two of which have a common finite etale covering.

Via Theorem 0.5 in \cite{Zuo} and Theorem 0.9 in \cite{Moller}, the Mumford curves play an important role among Shimura curves of Hodge type.

Our notations and main theorem are as following: 

Let $k$ be an algebraically closed field of characteristic $p$, $C$ be a proper smooth curve over $k$ with absolute Frobenius $\sigma$ and $\pi: X \ra C$ be a nonisotrivial family of the $2^m$ dimensional principally polarized abelian varieties over $C$.

\begin{defn}

\begin{enumerate} [(a)]
\item A curve in $\cA_{2^m,1,n} \otimes \C$ is called a \textit{special Mumford curve} if it is the image of a Mumford curve in $\cA_{2^m,1,n}\otimes \C$ induced by a universal family. 

\item The family $X \ra C$ is a \textit{weak Mumford curve over $k$} if the image of $C \ra \cA_{2^m,1,n}$ (induced by the family $X/C$) is (possibly an irreducible component of) a reduction of a special Mumford curve in $\cA_{2^m,1, n}\otimes \C$ .

\end{enumerate}
\end{defn}

Let $\cE$ be the Dieudonne crystal $R^1\pi_{\cri,*}(\cO_X)$ in $\cris(C/W(k))$ with Frobenius $F$ and Verschiebung $V$. 

\begin{thm} \label{main thm}
Assume $p>2$ and 
\begin{enumerate}
\item $X_c$ is ordinary for some closed point $c\in C$,
\item $\cE$ is irreducible and isomorphic to  $\cV_1 \otimes \cV_2 \otimes \cdots \cV_{m+1}$ as isocrystals,
\end{enumerate}
then $X \ra C$ is a weak Mumford curve over $k$.

If we further assume the Higgs field associated to $X \ra C$ is maximal,
then there exists another family of polarized abelian varieties $Y \ra C$ such that 
\begin{enumerate}[(a)]
\item there exists an isogeny $Y \ra X$ over $C$, 
\item $Y \ra C$ is a good reduction of a Mumford curve. 
\end{enumerate}

\end{thm}

\begin{rmk}
If condition (2) can be strengthened to $\cE \cong \cV_1 \otimes \cV_2 \otimes \cdots \cV_{m+1}$ as  \textit{crystals} for rank 2 irreducible $\cV_i$, then with maximal Higgs field, $X \ra C$ itself is already a Mumford curve over $k$.
 \end{rmk}

\subsection{Structure of the paper}
In Section \ref{Shimura curves of Mumford type} and \ref{The crystals}, we introduce the basic notions involved in \ref{main thm}.

In Section \ref{example}, we show that some good reductions of Mumford curves satisfy all three conditions in \ref{main thm}.

In Section \ref{crystal}, we study the $F$-crystal structure of $\cE$. The irreducibility of all $\cV_i$ implies a tensor decomposition of the Frobenius $F$ of $\cE$. Then each $\cV_i$ has an $F^f$-isocrystal structure for some integer $f$. Since $C$ intersects with the ordinary locus, the Newton slopes of $\cV_i$ are $\{0,f\},\{0,0\}, \cdots, \{0,0\}$, respectively. Assume $\cV_1$ is the ordinary one, and then we show that $F$ descends to $\cV_1$. So $\cV_1$ is actually a generically ordinary $F$-crystal. And the tensor $\cV_2 \otimes \cV_3 \otimes \cdots \cV_{m+1}$ is a unit root crystal. Then the isomorphism in (2) of \ref{main thm} is an isogeny between Dieudonne crystals: 
\[\psi: \cE \ra \cV_1 \otimes (\cV_2 \otimes \cdots \cV_{m+1}).\]

By \ref{equiv on curve}, the isogeny $\psi$ induces an isogeny between abelian schemes over $C$: 
\[f: Y \ra X\]
where $\cV_1 \otimes (\cV_2 \otimes\cdots \otimes \cV_m)$ is the Dieudonne crystal associated to $Y$. Let $G$ be the BT group corresponding to $\cV_1$.

In Section \ref{Higgs}, we prove the second part of \ref{main thm}, i.e. the case with maximal Higgs field. By a theorem in \cite{Xia} and changing of coordinates, we can show $Y\ra C$ has the maximal Higgs field. Then by \ref{the last one}, $Y \ra C$ is a good reduction of a Mumford curve. 

In Section \ref{without maximal Higgs field}, we prove the first part, i.e. the non-maximal Higgs field case. The family $Y \ra C$ induces a morphism $\phi: C \ra \cA_{2^m, d, n}\otimes k$. Then Theorem \ref{image with maximal Higgs field} shows that $\im \phi$ is a curve over which the universal family has the maximal Higgs field. Thus by \ref{the last one} $\im \phi$ is a reduction of a special curve in $\cA_{2^m, d, n} \otimes \C$. Then the first part of 
\ref{main thm} follows from \ref{lifting of isogeny}.  

\subsection*{Acknowledgment} I would like to express my deep gratitude to my advisor A.J.\,de Jong, for
suggesting this project to me and for his patience and guidance throughout its resolution. This
paper would not exist without many inspiring discussions with him.

\section{Shimura curves of Mumford type} \label{Shimura curves of Mumford type}

In this section, we briefly state the definition of Shimura curves of Mumford type. 

Let $\mathbb{S}$ be the Weil restriction of scalar $\mathrm{Res}_{\C/\R}$. A Shimura datum $(G, X)$ consists of a reductive group $G$ defined over $\Q$ and a $G(\R)-$conjugate class $X$ of a cocharacter $h: \mathbb{S} \ra G_\R$ such that 
\begin{enumerate}
\item For any $h$ in $X$, the Lie algebra $\frak{g}$ of $G_\R$, viewed as a conjugation representation of $\mathbb{S}$ via $h$, has the type $(1,-1), (0,0)$ and $(-1,1)$.
\item The adjoint action of $h(i)$ induces a Cartan involution on the adjoint group of $G_\R$.
\item The adjoint group of $G_\R$ does not have a factor $H$ defined over $\Q$ such that the projection of $h$ on $H$ is trivial.
\end{enumerate}

For any sufficiently small compact open subgroup $K$ of $G(\mathbb{A}_f)$, 
\[Sh_K(G, X)=G(\Q)\backslash X \times G(\mathbb{A}_f) /K\] is a complex algebraic variety. Take the inverse limit over $K$ and the resulting inverse system is called a \textit{Shimura variety}.

A large class of Shimura varieties admits an interpretation in terms of moduli of certain polarized abelian varieties. For instance, \textit{Shimura varieties of PEL type} parametrize polarized abelian varieties with a prescribed endomorphism ring. In such a class, we can take the moduli description as an equivalent definition of Shimura varieties. 

In \cite{famofav}, Mumford defines a class of Shimura varieties  to be the moduli scheme of polarized abelian varieties (up to isogeny) whose Hodge groups are contained in a prescribed Mumford-Tate group. We call this class \textit{Shimura varieties of Hodge type} which contains PEL type.

\subsection{Shimura curve of Mumford type over $\C$}
Let $K$ be a totally real field of degree $m+1$ and $D$ be a quaternion division algebra over $K$ which splits only at one real place and $\rm{Cor}_{K/\Q}(D)=M_{2^{m+1}}(\Q)$. In this case $D\otimes_\Q \R \cong \mathbb{H} \times \cdots \times \mathbb{H} \times M_2(\R)$.

Let $\bar \ $ be the standard involution of $D$, and let 
\[Q=\{x\in D^* | x\bar x=1\}. \]
Then $Q$ is a simple algebraic group over $\Q$ which is the $\Q-$form of the real algebraic group 
\[SU(2)^{\times m} \times SL(2,\R).\] 
Since $\rm{Cor}_{K/\Q}(D)=M_{2^{m+1}}(\Q)$, $Q$ admits a natural $2^{m+1}$ dimensional rational representation $V$ whose real form is 
\[\r: SU(2)^{\times m} \times SL(2) \ra SO(2^m)\times SL(2) \text{ acting on }  \R^{2^{m+1}}. \]
Note $Q_\C=SL(2,\C)^{ \times m+1}$. Then $V_\C$ is the tensor of $m+1$ copies of standard representation $\C^2$ of $SL(2,\C)$: 
\begin{equation} \label{the tensor decomposition of V}
V_C=\C^2 \otimes \C^2 \cdots \otimes \C^2\ \ \ (m+1 \text{ factors})
\end{equation} 

Let \begin{align*}
h: \mathbb{S}_{m}(\R) \ra& Q(\R)\\ e^{i\th} \mapsto&I_{2^m} \otimes \begin{pmatrix} \cos \th & \sin \th \\-\sin \th & \cos \th \\ \end{pmatrix}.
\end{align*} 
Then $(Q,h)$ defines a Shimura datum. So is $(\r(Q), \r \circ h)$. Generically $\r(Q)$ is the Hodge group of $V$. 

Let $\stab(h)\subset Q_\R$ be the stabilizer of $h$. Then $\stab(h)$ is a maximal compact subgroup of $Q_\R$ and hence congruent to $SO(2) \times SU(2)^{\times m}$. So $Q_\R/\stab(h)\cong Sp(1,\R)/SO(2, \R)\cong \frak{h}$ the upper half plane. Since $\ker \r\subset \stab(h)$, we have 
\[\r(Q)/\stab(\r\circ h)=Q_\R/\stab(h)=\frak{h}.\]
Let $\G \subset Q_\R$ be an arithmetic subgroup such that $\G$ acts freely and properly discontinuous on $\frak{h}$. Note $\ker \r \subset Z(Q)$ and then it fixes $h$, $\G\hookrightarrow \r(Q(\R)).$ 

We call the Shimura curves corresponding to the datum $(Q, h)$ \textit{the Shimura curves of Mumford type} (or \textit{Mumford curves}, for simplicity). If we don't specify level structures, Mumford curves mean a family of Shimura curves, any morphism between two of which is finite and etale.

By \cite{Zuo}, we know any family of polarized abelian varieties over a smooth proper curve with maximal Higgs field is a Mumford curve, up to taking powers and isogeny. Combining with following theorem, we know that in order to define any Shimura curves of Hodge type in positive characteristic, it is necessary to define Mumford curve in char $p$. 

\begin{thm}(\cite[0.9]{Moller}) \label{Moller}
Any smooth Shimura curve of Hodge type with the universal family has the maximal Higgs field.
\end{thm}

\section{The crystals associated to a Barsotti-Tate group} \label{The crystals}
 We explain the concepts involved in \ref{main thm} and state some results on crystals, Barsotti-Tate groups and Tannakian categories which we will use later. 
\subsection{}
The curve $C/k$ in \ref{main thm} has a natural crystalline site $\cri(C/W(k))$, induced from $(k, W(k), p)$. The higher direct image $\cE=R^1\pi_{\cri,*}(\cO_X)$ of the abelian scheme $\pi: X\ra C$ is a crystal in locally free sheaves. 

\begin{defn}
A \textit{Dieudonne crystal} over $\cri(C/W(k))$ is a triple $(\cE,F,V)$ where
\begin{enumerate}
\item $\cE$ is a crystal in locally free sheaves,
\item  $F: \cF^\sigma \ra \cF$ and $V: \cF \ra \cF^\sigma$ are homomorphisms between crystals such that $F\circ V=p. \mathrm{Id}_\cF$, $V \circ F=p. \mathrm{Id}_\cF.$
\end{enumerate}
\end{defn}

In particular, $\cE$ is a Dieudonne crystal. 

Choose an arbitrary lifting $\tilde C$ of $C$ to $W(k)$. The category of crystals in locally free sheaves in $\cri(C/W(k))$ is equivalent to the category of vector bundles with an integrable connection on $\tilde C$. In particular, choosing an open affine subset $U\subset C$ and a lifting $\tilde U$ of $U$, we have a lifting of Frobenius $\tilde \s$ over $\tilde U$. 

In the rest of the paper, by crystal, we mean a crystal in locally free sheaves. Therefore an $F$-isocrystal on $\cri(U/W(k))$ corresponds to a triple $(M, \nabla, F)$, a sheaf of module on $\tilde U$ with an integrable connection and Frobenius $F: M^{\tilde \s} \ra M$.

\subsection{} 
A Barsotti-Tate (BT) group $G$ over $C$ is a $p-$divisible, $p-$torsion and the $p-$kernel is a finite locally free group scheme.  Each $p^i-$kernel $G[p^i]$ is a truncated BT group. By \cite{BBM}, the crystalline Dieudonne functor $\DD(G)=\scr{E}xt^1(\underline{G}, \cO_C)$ associates a Dieudonne crystal over $\cri(C/W(k))$ to a BT group.  And $\DD(G[p])=\DD(G)_C$ admits a filtration 
\[0\ra \o_G \ra \DD(G)_C\ra \a_G \ra 0.\] 

In the context of the Theorem \ref{main thm}, $\cE=\DD(X[p^\infty])$ and the filtration on $\cE_C$ is just the Hodge filtration of $X/C$: $\o=\pi_* \O_{X/C}, \a=R^1\pi_*(\cO_X)$. In particular, $\cE_C$ has the Gauss-Manin connection and it induces a $\cO_C-$linear map, called Higgs field: 
\[\theta: \o \ra \a\otimes \O^1_C\] which is related to Kodaira-Spencer map. The Higgs field can be defined in alternative ways: one can use
\[\xymatrix{
\o \ar[r] \ar@{-->}[drr] & \cE_C \ar[r] \ar[d]^\nabla & \a \\
&\cE_C \otimes \O^1_C \ar[r] & \a\otimes \O^1_C.
}
\] The other is from the long exact sequence of 
\[0\ra \pi^*\O_C \ra \O_X \ra \O_{X/C} \ra 0\] and the boundary map $\cdots \ra \pi_*\O_{X/C} \xrightarrow{\partial} R^1\pi_* \cO_X \otimes \O^1_C \ra \cdots$ gives the Higgs field.
Condition (3) in \ref{main thm} just means the map $\theta$ is isomorphic.

\begin{thm}(\cite[Main Theorem 1]{deJ} ) \label{equiv on curve} The category of Dieudonne crystals over $\cri(C/W(k))$ is anti-equivalent to the category of BT groups over $C$. 
\end{thm}

\subsection{Tannakian category}

\begin{defn} Let $L$ be a field of characteristic 0.
A Tannakian category $T$ (over $L$) is a $L-$linear neutral rigid tensor abelian category with an exact fiber functor $\o: T\ra \rm{Vect}_L $. 
\end{defn}

\begin{thm} (\cite[Theorem 2.11]{Deli}) \label{Deli}
For any Tannakian category $T$, there exists an $L-$algebraic group $G$ such that $T$ is equivalent to $\Rep_L(G)$ as tensor categories. 
\end{thm}

\begin{example} (\cite[VI 3.1.1, 3.2.1]{Saa} ) \label{T and iso}
Inverting $p$ in the category $\cris(C)$, we obtain $\iso(C)$. The category $\iso(C)$ forms a Tannakian category, with the obvious fiber functor.  Hence there exists a $B(k)-$affine group scheme $G_\univ$ such that the following two categories are equivalent. 
\[\{\text{finite locally free isocrystals on } C/W(k) \} \longleftrightarrow \Rep_{B(k)}(G_\univ).\]
An object $\cF'$ in $\iso(C)$ is called effective if it is from an object $\cF$ in $\cris(C)$,i.e. $\cF' =\cF \otimes B(k)$. For any morphism $f: \cF\otimes B(k) \ra \cG \otimes B(k)$ between effective objects in $\iso(C)$, there exists $m\in \Z$ such that $p^m f: \cF \ra \cG$ is a morphism in $\cris(C)$. 

Note different from (\cite[VI 3.1.1, 3.2.1]{Saa}), $\iso(C)$ denotes just the isocrystals, not the $F$-isocrystals. So $G_\univ$ is an affine group scheme over $B(k)$.
\end{example}

We will need another result later. 
\begin{prop} \label{the inclusion}
For any Tannakian category $T$ and $W, V \in T$, let $<W>$ denote the Tannakian subcategory generated by $W$, with Tannakian group $G_W$. Similarly, $<V>=\Rep_L(G_V),<W,V>=\Rep_L(K)$. Then there exists a natural injection $K \hookrightarrow G_W\times G_V$.
\end{prop}

\begin{proof}
Since $W, V\in \Rep(K)$, by (\cite[2.21]{Deli}), $K$ admits surjections onto $G_W$ and $G_V$. Then $K$ admits a map $K \ra G_W\times G_V$. The induced morphism $\Rep(G_W \times G_V) \ra \Rep(K)$ satisfies (\cite[2.21(2)]{Deli}). So the map is injective. 
\end{proof}

\section{Example} \label{example}

In this section, we show that some good reductions of Mumford curves satisfy the conditions in \ref{main thm}. This is the main theorem in \cite{JX}.

\begin{thm} (\cite[Theorem 1.2]{JX}) \label{my thm}

For infinitely many prime $p$, there exists a family of principally polarized abelian varieties of $2^m$ dimension over a smooth proper curve $\tilde f: \tilde X \ra \tilde C$ over $W(k)$ such that 
\begin{enumerate}
\item $(\tilde X \xrightarrow{\tilde f} \tilde C)\otimes \C $ is a Mumford curve,
\item the reduction of $\tilde X \ra \tilde C$ at $k$ is generically ordinary,
\item the Dieudonne crystal $\cE\cong \cV \otimes \cT$ where $\cV$ is a Dieudonne crystal of rank 2 with maximal Higgs field and $\cT$ is a unit root crystal of rank $2^m$.
\end{enumerate}
\end{thm}

\begin{rmk}
Sheng Mao and Kang Zuo (\cite{Zuo2}) study the Newton polygon of good reductions of Mumford curves. Their results show the reduction of a Mumford curve over infinitely many primes intersect with ordinary locus. 
\end{rmk}

By Proposition 6.3 in \cite{JX}, we know that $\cT$ admits a tensor decomposition as rank 2 irreducible crystals:
$\cT\cong \cV_2 \otimes \cV_3 \cdots \otimes \cV_{m+1}.$ The Frobenius is a morphism permuting $\cV_i$. Specifically, there exists a permutation $s\in S_m$, rank 1 crystals $\cL_i$ with $\otimes_i \cL_i \cong \cO_C$ and isomorphisms 
\[\phi_i: \cV^\s_i \otimes B(k) \ra \cV_{s(i)}\otimes \cL_i \otimes B(k)\]  such that $F=\otimes_i \phi_i.$


 Then $\cT\cong \cV_2 \otimes \cdots \otimes \cV_{m+1}$ as isocrystals and hence $\cE \cong \cV \otimes \cV_2 \cdots \otimes \cV_{m+1}$ as isocrystals. So the reduction $X \ra C$ of $\tilde X \ra \tilde C$ satisfies (1) and (2) in \ref{main thm}.

For condition (3), by \ref{Moller}, the universal family over a Shimura curve of Hodge type has maximal Higgs field, and hence the Higgs field of the special fiber $X/C$ is also maximal. 

\section{the structure of $\cE$ as crystals} \label{crystal}

In this section, we choose an $F$-crystal model of $\cV_1 \otimes \cV_2 \otimes \cV_3\cdots \otimes \cV_{m+1}$. 

Note $\cE$ is an $F$- isocrystal. So \[F: \cV^\s_1\otimes \cV^\s_2 \otimes \cV^\s_3 \cdots \otimes \cV^\s_{m+1} \ra \cV_1 \otimes \cV_2 \otimes \cV_3 \cdots \otimes \cV_{m+1}\]
is a morphism in the category of isocrystals.

By condition (2) in \ref{main thm}, each $\cV_i$ is irreducible as isocrystals. Let $H_i$ be the Tannakian group associated to $\cV_i$ in $\Rep(G_\univ)$. 

\begin{lemma}
$H_i$ has the form $SL(2)\times \m_m$ for some integer $m$.
\end{lemma}

\begin{proof}
Since $\cV_i$ are all irreducible isocrystals, each of the group $H_i$ admits a faithful irreducible representation of dimension 2. Thereby $H_i$ is a reductive group over $B(k)$. Base change to $\bar B(k)$, $H_i \otimes \bar B(k)$ must be a central extension of $SL(2, \bar B(k))$.  Since $H_i \hookrightarrow GL(2, B(k))$ and the field extension is faithfully flat,  $H_i$ is $SL(2, B(k)) \times \m_m$ or $GL(2, B(k))$.

Since $X \ra C$ is principally polarized, we have $\cE \cong \cE^\vee$ up to some twist.  Thereby as representations, $\wedge^8 E \cong \wedge^8 E^\vee$ and hence $\wedge^2 V_1 \otimes \cdots \wedge^2 V_{m+1}$ is 2-torsion. Therefore $\det H_i$ can not be $\mathbb{G}_m$, i.e. $H_i$ can not be $GL(2)$.
\end{proof}

As an irreducible representation of $SL(2) \times \m_m$, $V_i \cong V'_i \otimes L_i$ where $V'_i$ the standard representation of $SL(2)$ and $L_i$ an irreducible representation of $\m_m$.  Thus we can adjust each $V_i$ such that $H_i=SL(2) $ for $1\leq i \leq m$ and only $H_{m+1}=SL(2) \times \m_m$ . 

Since $\cE$ is irreducible, $\otimes \cV_i$ corresponds to an irreducible $SL(2)^{\times m+1}\times \m_m-$representation.  Then we can modify the proof of Proposition 6.3 in \cite{JX} in this case and obtain a similar result: there exist 
\[\phi_i: \cV^{\s}_i  \ra \cV_{s(i)} \otimes \cL_i \]
 such that $F=\otimes \phi_i$. 
We can find an $f$ such that $s^f=\mathrm{Id}$. Then $F^f$ can be decomposed to the tensor product of 
\[\phi'_i: \cV^{\s^f}_i \ra \cV_i \otimes \sL_i \] for some $\sL_i$.  

\begin{prop} \label{surjectivity} (\cite[Proposition 7.1]{JX})
The group endomorphism $\s^{f\,*}-\mathrm{Id}$ of $\Pic(\bar C/W(k)_\cri)$ is surjective. 
\end{prop}

We can find $\cL'_i$ such that $\sL^{-1}_i=((\s^f)^*-\mathrm{Id})(\cL'_i)$, then $\phi'_i$ induces a morphism 
\[\cV^{\s^f}_i\otimes \cL'^{\s^f}_i \ra \cV_i \otimes \cL'_i.\] 
Replace $\cV_i$ by $\cV_i \otimes \cL'_i$ and let us still denote it as $\cV_i$. Therefore 
\[\phi'_i: \cV^{\s^f}_i \ra \cV_i .\]So each $\cV_i$ is an $F^f-$isocrystal. 

By condition (1) of \ref{main thm}, as an $F^f-$isocrystal, $\cE$ has generic slopes $\{0\times 2^m, f\times 2^m\}$.
\begin{lemma} \label{the slopes}
Generically, the slopes of $\cV_i$ are all $\{0,0\}$ except one $\{0,f\}$.
\end{lemma}

\begin{proof}
Use $\cV_{i,\eta}$ to denote the restriction of $\cV_i$ to the crystalline site $\cri(W(\eta^-)/\eta^-)$ where $\eta^-$ is the geometric generic point. Then slopes of $\cV_{1,\eta} \otimes \cV_{2,\eta} \otimes \cV_{3,\eta} \cdots \otimes \cV_{{m+1}, \eta}$ are $\{0\times 2^m, f\times 2^m\}$.  Now we can compute directly the slopes of each $\cV_i$. For notational simplicity, here we just show the case $m=2$.

Assume $\{a_1, a_2\}, \{b_1, b_2\}, \{c_1, c_2\}$ are slopes of $\cV_1, \cV_2, \cV_3$, respectively.  Further, we could assume the slopes of each $\cV_i$ are non-negative. Then the slopes of $\cE$ are 
\[a_i+b_j+c_k, 1\leq i,j,k\leq 2.\]
If $a_1 \leq a_2$ and $b_1 \leq b_2$, then we can adjust that $a_1=b_1=0$. Then $\{c_1, c_2 \}=\{0, f\}$ or $\{0, 0\}$. The former case forces all $a_i, b_j$ to be zero, while the latter case implies $\{a_2, b_2\} =\{0, f\}$. 
\end{proof}



 
\begin{prop} \label{s(j)=j}
$s(j)=j$ for some $j\in \{1,2,3\cdots, m+1\}$.
\end{prop}

\begin{proof}
From \ref{the slopes},  we can assume the slopes of $\cV_1$, restricting to $\cri(W(\eta^-)/\eta^-)$, are $\{0, f\}$. If $s(1)\neq 1$, say $s(2)=1$, then over $W(\eta^-)$, $\im \phi_2=\cV_{1,\eta}$ has only one-dimension subspace of slope 0. Iterate the Frobenius $f$ times and note $(\otimes \phi_i)^f=F^f=F_1 \otimes F_2 \otimes F_3 \cdots \otimes F_{m+1}$. Thus $\cV_{2,\eta}=\im F_2$ has at most one dimensional subspace of slope 0, contradicting to $\cV_2$ with slopes $\{0,0 \}$. So $s(1)=1$.
\end{proof}

Without loss of generality, assume $s(1)=1$. Then $\phi_1:\cV^\s_1 \ra \cV_1 \otimes \cL_1$.

By \ref{surjectivity}, we can find $\cL'$ such that $\cL^{-1}_1=\Pp(\cL')$, then $\phi_1$ induces a morphism \[\cV^\s_1\otimes \cL'^\s \ra \cV_1 \otimes \cL' .\] 
Replace $\cV_1$ by $\cV_1 \otimes \cL'$ and let us still denote it as $\cV_1$. Therefore 
\[\phi_1: \cV^\s_1 \ra \cV_1.\]

Then $\cV_1$ and $\cV_2\otimes \cV_3 \cdots \otimes \cV_{m+1}$ are actually $F$-isocrystals. And $\cE \cong \cV_1 \otimes (\cV_2 \otimes \cV_3 \cdots \otimes \cV_{m+1})$ as $F$-isocrystals. Since $\cE$ is further a Dieudonne crystal, the Verschiebung $V$ also can be decomposed to $V_1\otimes V_2$ where both of $V_1: \cV_1\ra \cV^\s_1, V_2: \cV_2 \otimes \cV_3 \cdots \otimes \cV_{m+1} \ra \cV^\s_2 \otimes \cV^\s_3 \cdots \otimes \cV^\s_{m+1}$ are isomorphisms.

Since all the slopes of $\cV_i$ are nonnegative, we can choose an $F^f-$crystal model for each $\cV_i$( see Appendix \ref{crystal model}). Since as an $F$-isocrystal, $\cV_1$ has slopes $\{0,1\}$, there exists $V_1: \cV_1 \ra \cV^\s_1$ such that $F_1 \circ V_1 =V_1 \circ F_1 = p$. Thereby we can choose the descent of Frobenius and Verschiebung so that $\cV_1$ is a Dieudonne crystal. For simplicity, we still denote the corresponding crystal as $\cV_i$. Though the isomorphism in \ref{main thm} (2) may not be true in the level of crystals, we still can choose 
 \[\psi: \cV_1 \otimes \cV_2 \otimes \cV_3 \cdots \otimes \cV_{m+1}\ra \cE\]
injectively between $F$-crystals and $\psi[\frac{1}{p}]$ is the isomorphism. Multiplying some power of $p$ if necessary, we can assume $\psi\mod p$ is not zero and $\cE/\im \psi$ is $p-$torsion. From the information of the slopes, we have $\cV_1$ is generically ordinary and $\cV_2 \otimes \cV_3 \cdots \otimes \cV_{m+1}$ is a unit root $F$-crystal.
For simplicity, we use $\cT$ to denote the $F$-crystal $\cV_2 \otimes \cV_3 \cdots \otimes \cV_{m+1}$ in the rest of the paper. 
\begin{rmk}
We observe that if we change the condition (2) in \ref{main thm} to the following
\begin{center}
$\cE \cong \cV_1 \otimes \cdots \otimes \cV_{m+1}$ as $F^f-$isocrystals for some integers $f$ where $\cV_i$ are all irreducible of rank 2,
\end{center} and further, $X \ra C$ is defined over a finite field,  then we can consider the (neutral) Tannakian category of $F^f-$isocrystals instead of isocrystals. Mimic the above arguments and we can achieve the identical result that $\cE$ is isogenous to a tensor decomposition of $\cV\otimes \cT$. 
\end{rmk}

\section{Maximal Higgs field} \label{Higgs}
In this section, we prove the second part of \ref{main thm}, the case that $X \ra C$ has maximal Higgs field.

By \ref{equiv on curve}, $\cV_1$ corresponds to a BT group over $C$. In particular, it admits a Hodge filtration: $0\ra \sL_1 \ra \cV_{1|C} \ra \sL_2 \ra 0$  with Higgs field $\theta_1: \sL_1  \ra \sL_2  \otimes \O^1_{C}$.  Since $\cT$ is a unit root crystal, the Hodge filtration is trivial. Therefore the Hodge filtration of the Dieudonne crystal $\cV_1 \otimes \cT $ is 
\[0\ra \sL_1 \otimes \cT_{|C} \ra (\cV_1\otimes \cT)_{|C} \ra \sL_2 \otimes \cT_{|C} \ra 0.\] 
The associated Higgs field $\theta'$ is $\theta_1 \otimes \text{id}_\cT$ where $\theta_1$ is the Higgs field of $\cV_1$.  
 
\subsection{Hodge and conjugate filtrations}For any Dieudonne crystal $\cF$ over $\cri(C/W(k))$, $\cF_C$ admits two filtrations: Hodge filtration and conjugation filtration. The relation between the two filtrations is shown in \cite[Proposition 2.5.2]{deJ}: conjugation filtration is given by the kernel of $F_C$: $\o^{(p)}\subset \cF^{(p)}_C$, which the Frobenius pullback of the Hodge filtration $\o\subset \cF$. Since the isogeny $\p: \cV_1 \otimes \cdots \otimes \cV_{m+1} \ra \cE$ is compatible with the Frobenius, $\psi$ induces actually a morphism between Hodge filtrations: 
\[\xymatrix{
\sL_1 \otimes \cT_{C} \ar[r]^{\p_1} \ar[d] & \o \ar[d]\\
 {(\cV_1)}_{C} \otimes \cT_{C} \ar[r] \ar[d] & \cE_{C} \ar[d]\\
\sL_2 \otimes \cT_{C} \ar[r]^{\p_2} & \a. \\
}\] 
Similarly, $\p_{C}$ also preserves the conjugate filtration: 
\[\xymatrix{
\sL^{(p)}_2 \otimes \cT_{C} \ar[d] \ar[r]^{\p^{(p)}_2} & \a^{(p)} \ar[d]\\
(\cV_1)_{C} \otimes \cT_{C}  \ar[d] \ar[r] & \cE_{C} \ar[d]\\
 \sL^{(p)}_1 \otimes \cT_{C} \ar[r]^{\p^{(p)}_1} & \o^{(p)}\\
}
\]
Further, since $\p$ preserves the connections, it induces a morphism between Higgs fields.
\begin{equation}\label{Higgs diagram}
\xymatrix{
\sL_1 \otimes \cT \ar[r]^{\p_1} \ar[d]_{\theta_1\otimes \mathrm{Id} \cong \theta'} & \o \ar[d]^\cong\\
\sL_2 \otimes \cT \otimes \O^1_{C} \ar[r]^{\p_2 \otimes \text{id}} & \a \otimes \O^1_{C}
}\end{equation}

\begin{lemma} \label{from a to a}
The Hodge and conjugate filtrations of $\cE$ induce a morphism $\a^{(p)} \ra \a$ which is generically surjective over $C$.
\end{lemma}

\begin{proof}
Combine the two filtrations in one diagram: 
\[\xymatrix{
&\a^{(p)} \ar[d] \ar@{-->}[dr]^ h & \\
\o \ar[r] & \cE \ar[d] \ar[r] & \a \\
& \o^{(p)}&
}\] For any $x\in C$, let the stalk $\a^{(p)}_x$ generate by $\{a_1 \otimes 1, \cdots, a_{2^m} \otimes 1\}$ where $a_i \in \a_x$.  Since $\nabla(a_i \otimes 1)=0$ and the Higgs field is maximal, $l(a_i \otimes 1)$ are linearly independent in $\a_x$. Otherwise, some section $a \otimes 1 $ is a local section of $\o$ and due to maximal Higgs field, any local section of $\o$ is not horizontal, contradiction. Therefore $a_i \otimes 1$ generate the stalk $\a_x$ generically. Therefore $l$ is generically surjective. 
\end{proof}

\subsection{The Higgs field is nonzero}

\begin{lemma} \label{nonzero degree}
If Higgs field $\theta_1$ of $\cV_1$ is zero, then $\deg \sL_2 \neq 0$.
\end{lemma}
\begin{proof}
We prove it by contradiction. Consider the Hodge and conjugate filtration of $(\cV_1)_{C}$: 
\[\xymatrix{
&\sL^p_2 \ar[d]^F \ar[dr]^l \\
\sL_1 \ar[r] & (\cV_1)_{C} \ar[r] \ar[d]^V & \sL_2 \\
& \sL^p_1 .\\
}
\] Choose any open affine subset $U\subset C$ over which $\sL_1$ is free. Let $t$ be a generator of $\sL_1$ over $U$. Then if $l(t)=0$, then $F(t)\in \sL_1(U)$. Since $\p_1=0$, then for any section $s\in \cT(U)$, $\p(t \otimes s)=0.$ However, $\p(t\otimes s)=\p^{(p)}_2(t \otimes s)$ and hence $\p_2$ is zero over $U$. Since $\im \p_2$ is torsion-free, $\p_2=0$ globally, contradiction. Thereby $l$ is injective and it induces an injection $\cO_{C} \hookrightarrow \sL^{1-p}_2$. In particular, $\sL^{1-p}_2$ is effective. If $\deg \sL_2=0$, then $\sL^{1-p}_2\cong\cO_{C}$ and hence $l$ is an isomorphism. 

Note $\cT$ is a unit root crystal. The isomorphism $l$ induces isomorphism $\sL^p_2 \otimes \cT^{(p)} \ra \sL_2 \otimes \cT$. It implies $Y$ is everywhere ordinary over $C$. However, $C$ is proper which can not be contained in the ordinary locus of $\cA_{2^m,1,k}$.
\end{proof}

Since $\deg (\cV_1)_{C}=0$, we have $\deg \sL_1 \neq 0$ as well.

\begin{prop} \label{Frobenius pullback}
If $\theta_1=0$, then there exists a rank 2 Dieudonne crystal $\cV$ such that $\cV_1=\cV^\s$ and the $\cV_1 \otimes \cT \ra \cE$ factors through the Frobenius map $F_\cV: \cV_1\otimes \cT \ra \cV \otimes \cT $ of $\cV\otimes \cT$.
\end{prop}

\begin{proof}
The Higgs field of $\cV_1$ is trivial and then $\p_1=0$. By (\cite[Theorem 6.1]{Xia}), $\cV_1$ is the Frobenius pullback of some Dieudonne crystal $\cV$.  Let $F_\cV, V_\cV$ denote the Froebnius and Verschiebung of $\cV$.  Since $\p_1=0$, the image of $\p \circ V_\cV$ is contained in $p\cE$.  Thereby $\frac{\p\circ V_\cV}{p}: \cV \otimes \cT \ra \cE$ is a well defined morphism. Since $\frac{\p\circ V_\cV}{p} \circ F_\cV=\p$, $\p$ factors through $F_\cV$: 
\[\xymatrix{
&\cV \otimes \cT \ar[dr]&\\
\cV_1 \otimes \cT \ar[ur]^F \ar[rr]^\p&& \cE.
}\]\end{proof}

\begin{rmk} \label{nonzero}
 If the Higgs field of $\cV$ is still trivial, then we repeat the process of \ref{Frobenius pullback}. Note the Hodge filtration $\sL \subset (\cV)_{C}$ of $\cV$ satisfies $\sL^p=\sL_1$.  Since by \ref{nonzero degree} $\deg \sL_1 \neq 0$, $|\deg \sL| < |\deg \sL_1|$. After a finite steps, the degree of the sub line bundle from the Hodge filtration is no longer divisible by $p$ and hence the Gauss Manin connection no longer preserves the Hodge filtration. The Higgs field is therefore nonzero. Then we can assume $\cV \otimes \cT$ with nontrivial Higgs field $\theta=\theta_\cV \otimes \text{id}_\cT$. 
\end{rmk}

\subsection{The Higgs field is maximal} \label{the Higgs field is maximal}
In the following, we will show the Higgs field $\theta$ of $\cV \otimes \cT$ is isomorphic. 

Let $0\ra \sL \ra (\cV)_{C} \ra \sL' \ra 0$ be the Hodge filtration of $(\cV)_{C}$.  By \ref{equiv on curve}, the Dieudonne crystal $\cV \otimes \cT$ corresponds to a BT group $G\otimes U$ over $C$ where $G$ is a height 2, generically ordinary BT group with nontrivial Kodaira Spencer map and $U$ is a height $2^m$ etale BT group. The BT groups associated to $\cV \otimes \cT$ and $\cE$ are isogenous and therefore $G \otimes U$ also comes from an abelian scheme $Y$ over $C$ which is isogenous to $X$. Let \[f: Y \ra X\] be the isogeny.

Note $Y[p^\infty] \cong G \otimes U$.  Let \[K=\ker f.\]	
 Then $K$ is a finite flat group scheme. Let $G[p^n] \otimes U[p^n]$ be the smallest truncated BT group containing $K$. 

Then we can choose a finite etale covering of $C$ such that the pullback of $U[p^n]$ is trivial. Then the pullback of $Y[p^n]$ is $G[p^n]^{\times 2^m}$ and $Y$ has maximal Higgs field if and only if the pull back of $Y$ does. For simplicity, we still denote the etale covering as $C	$.

\begin{lemma} \label{G_n}
For any $n$, any nontrivial subgroup scheme of $G_{n,\eta}$ contains $\m_p$.
\end{lemma}

\begin{proof}
Let $H$ be a nontrivial subgroup scheme of $G_{n,\eta}$.  We can reduce $n$ so that $H$ is not contained in $G_{n-1}$.  
\[\xymatrix{
&H \ar@{^{(}->}[d] \ar[dr] & \\
G_{n-1,\eta} \ar[r] & G_{n,\eta} \ar[r]^{p^{n-1}} & G_{1, \eta}.\\
}\]
Then $H$ has a nontrivial image in $G_{1,\eta}$ which is flat (over the field). By \cite[Proposition 2.5.2]{Xia}, this image contains $\m_p$. Note $p^{n-1}$ is an endomorphism of $H$. Thus $\m_p \subset H$. 
\end{proof}

Now we prove that $K \cap (1^{\times {2^m-1}} \times G_n )=1$.  The key observation is that since $\dim \Lie K \leq {2^m-1}$( see the proof of Theorem \ref{first three}), by changing of coordinates, we can make $K_\eta$ be contained in the product of first three factors $G^{\times {2^m-1}}_{n,\eta}$ of $G^{\times 2^m}_{n, \eta}$.

Since $G$ is generically ordinary, the filtration 
\[0 \ra G_{n,\eta,\mult} \ra G_{n, \eta} \ra G_{n, \eta, \et} \ra 0.\] Since the group schemes are defined over a field, we may assume the multiplicative part is isomorphic to $\m_{p^n}$, the etale part isomorphic to $\Z/p^n\Z$. Since over a field, all the group scheme involved are automatically flat. We start with the following proposition

\begin{prop} \label{m}
Over $k$, for any $n$, if $H$ is a subgroup of $\m^{\times {2^m}}_{p^n}$ with $\dim \Lie H \leq {2^m-1}$, then via changing of coordinates, $H$ is contained in  the product of first $2^m-1$ factors $\m^{\times {2^m-1}}_{p^n}$.
\end{prop}

\begin{proof}
We consider the case $\dim \Lie H ={2^m-1}$ and we prove it by induction. 

If $n=1$, then $H \subset \m^{\times {2^m}}_p$ such that $H$ has height ${2^m-1}$. Since $\m_p$ is simple, $H\cong \m^{\times {2^m-1}}_p$ and one of projections of $H$ to the product of any $2^m-1$ factors $H \ra \m^{\times {2^m-1}}_p$ is injective.  Assume the projection to the first $2^m-1$ factors is isomorphic. Then by change of coordinates, we can put $H = \m^{\times {2^m-1}}_p\subset \m^{\times {2^m}}_p$.

If the proposition is true for $k-1$, then for $n=k$, after changing of coordinates, we have $H\cap \m^{\times {2^m}}_{p^{k-1}}\subset \m^{\times {2^m-1}}_{p^{k-1}}$. Let pr$_{2^m}$ be the projection of $H$ to the $2^m-$th factor and $K_{2^m}=\ker \text{pr}_{2^m}$
\[K_{2^m} \hookrightarrow H \xrightarrow{\text{pr}_{2^m}} \m_{p^k}.\]
We may assume pr$_{2^m}$ is nontrivial. Obviously $K_{2^m}\subset \m^{\times {2^m-1}}_{p^k}$. Since $pH \subset H\cap \m^{\times {2^m}}_{p^{k-1}}\subset \m^{\times {2^m-1}}_{p^{k-1}}$, $\im \text{pr}_{2^m} \subset \m_p$. And $p^{k-1} K_{2^m} \subset \m^{\times {2^m-1}}_p$ is a flat sub group scheme with height $h$. Since pr$_{2^m}$ is nontrivial, $h\leq 2^m-2$. Note the reduction 
\[\Aut(\m^{\times {2^m-1}}_{p^k}) \cong GL_{2^m-1}(\Z/p^k\Z) \ra \Aut(\m^{\times {2^m-1}}_p)\cong GL_{2^m-1}(\F_p)\] is surjective. Thus we can apply some automorphism of $\m^{\times {2^m-1}}_{p^k}$ to implement the coordinate change so that we can make $p^{k-1} K_{2^m}$ be exactly the product of first $h$ factors $\m^{\times h}_p$. 

Consider the following diagram
\[\xymatrix{
&H \ar[r] \ar@{^{(}->}[d] & H/H\cap \m^{\times {2^m}}_{p^{k-1}} \ar@{^{(}->}[d]\\
\m^{\times {2^m}}_{p^{k-1}} \ar[r] & \m^{\times {2^m}}_{p^k} \ar[r]^{p^{k-1}} & \m^{\times {2^m}}_p .\\
}\]
Since $H\cap\m^{\times {2^m}}_{p^{k-1}}= K_{2^m} \cap \m^{\times {2^m-1}}_{p^{k-1}}$, the morphism pr$_{2^m}$ factors through $H/ H\cap\m^{\times {2^m}}_{p^{k-1}}$ and hence $\m^{\times {2^m}}_p \supset H/H\cap\m^{\times {2^m}}_{p^{k-1}} \ra \m_p $ surjectively, induced by pr$_{2^m}$. The kernel of this surjective map, denoting as pr$'_{2^m}$,  is exactly $K_{2^m}/K_{2^m} \cap \m^{\times {2^m-1}}_{p^{k-1}} \cong \m^{\times h}_p$. Therefore the graph of pr$'_{2^m}$ is a height 1 subgroup scheme in $\m^{\times ({2^m}-h)}_p$ and pr$'_{2^m}$ is actually the last factor projection of this subgroup scheme. There exists an element in $\Aut(\m^{\times {2^m}}_{p^k}) $ whose reduction mod $p$ transforms the height 1 subgroup scheme to be the first factor in $\m^{\times ({2^m}-h)}_p$, i.e. the  image of pr$_{2^m}$ is trivial. Note the automorphisms involved here are of the form 
$\begin{pmatrix}
1&&&\\
&1&&\\
&&1&\\
p^{k-1} a&p^{k-1}b&p^{k-1}c&1
\end{pmatrix}$ 
which fixes the first $2^m-1$ factors $\m^{\times {2^m-1}}_{p^k}$. 
Thereby $H$ is contained in $\m^{\times {2^m-1}}_{p^k}$.

For the easier cases $\dim \Lie H <{2^m-1}$, we can adjust above argument accordingly. 
\end{proof}

Now we come back to the original case. Firstly we need the following lemma which relates the \ref{m} and our target: 
\begin{lemma} \label{auto surjection}
The restriction $\Aut(G[p^n]_\eta) \ra \Aut(\m_{p^n})$ is surjective.
\end{lemma}

\begin{proof}
It follows from $\Aut(\m_{p^n})=(\Z/p^n\Z)^*$ and $(\Z/p^n\Z)^*$ naturally acts on any $p^n$ torsion group scheme, in particular, $G_{n, \eta}.$
\end{proof}

\begin{prop} \label{first three}
By changing coordinates, $K_{\eta}$ is contained in the product of the first $2^m-1$ factors $G^{\times {2^m-1}}_{n,\eta}$.
\end{prop}

\begin{proof}
Since the Higgs field $\theta_1$ is nonzero by \ref{nonzero}, $f: Y \ra X$ induces a nonzero map on the tangent bundles. Thereby the dimension of $\Lie K$ is at most ${2^m-1}$.
By \ref{m} and \ref{auto surjection}, we can assume $K_{\eta} \cap \m^{\times {2^m}}_{p^n} \subset \m^{\times {2^m-1}}_{p^n}$.
Let pr$_{2^m}$ be the projection $G^{\times {2^m}}_{n,\eta}\ra G_{n, \eta}$ to the $2^m$ factor and $K_{2^m}$ be the kernel of the restriction of $\text{pr}_{2^m}: K_{\eta} \ra G[p^n]_\eta$.
\[\xymatrix{
&K_{\eta} \ar@{^{(}->}[d] &\\
\m^{\times {2^m}}_{p^n} \ar[r] & G^{\times {2^m}}_{n,\eta} \ar[r] & \Z/p^n\Z^{\times {2^m}}.\\
}\]
Obviously $K_{\eta} \cap \m^{\times {2^m}}_{p^n} \supset K_{2^m} \cap \m^{\times {2^m-1}}_{p^{n}}$. Since $K_{\eta} \cap \m^{\times {2^m}}_{p^n} \subset \m^{\times {2^m-1}}_{p^n}$, pr$_{2^m}(K_{\eta} \cap \m^{\times {2^m}}_{p^n})=1$ and hence $K_{\eta} \cap \m^{\times {2^m}}_{p^n} \subset K_{2^m} \cap \m^{\times {2^m-1}}_{p^{n}}$. So $K_{\eta} \cap \m^{\times {2^m}}_{p^n} = K_{2^m} \cap \m^{\times {2^m-1}}_{p^{n}}$

 Thereby on one hand, $K_{\eta}/K_{2^m}$ is a quotient of an etale group scheme $K_{\eta} / (K_{\eta} \cap \m^{\times {2^m}}_{p^n}) \subset (\Z/p^n\Z)^{\times {2^m}}$ by the subgroup $K_{2^m}/(K_{2^m}\cap \m^{\times {2^m-1}}_{p^{n}})$. On the other hand, it is also a subgroup of $G[p^n]_\eta$. By \ref{G_n}, if it is nontrivial, $\m_p \subset K_{\eta}/K_{2^m} $, contradiction to etale. Thus $K_{\eta}=K_{2^m}$ and $\im \text{pr}_{2^m}$ is trivial.
\end{proof}

\begin{prop} \label{trivial intersection}
As subgroup schemes of $G^{\times 2^m}$, we can change the coordinates so that 
\[K \cap (1^{\times {2^m-1}} \times G_n )=1.\]
\end{prop}

\begin{proof}
By \ref{first three}, since $K_{\eta} \subset G^{\times {2^m-1}}_{n,\eta} \subset G^{\times 2^m}_{n,\eta}$ and $K, G_n$ are flat, taking the closure in $G^{\times 2^m}$ gives $K \subset G^{\times {2^m-1}}_n$. 
Thereby globally $K \cap (1^{\times {2^m-1}} \times G_n)=1$.
\end{proof}

\begin{thm}
The Higgs field $\theta$ of $\cV \otimes \cT$ is maximal.
\end{thm}

\begin{proof}
Note $f: Y \ra X$ induces a diagram 
\[\xymatrix{
\o \ar[r] \ar[d]^\cong & \sL_1 \otimes \cT  \ar[d]^{\theta\cong \theta_1 \otimes \text{Id}}\\
\a\otimes \O^1_C \ar[r]^{f^*} & \sL_2 \otimes \cT\otimes \O^1_C.
}\]
We have known that $\theta$ is not zero. Since $\theta_1: \sL_1 \ra \sL_2 \otimes \O^1_C$, over each fiber of $C$, $\theta$ is either zero or an isomorphism. It suffices to show $f^*$ is nonzero over each fiber.

Consider the composition
\[G_n \hookrightarrow G\odot H\cong Y[p^\infty] \ra X[p^\infty]\] over $C$ where the first arrow is the embedding to the fourth factor.  By \ref{trivial intersection}, this composition is injective. Therefore $\im f$ contains a generically ordinary height 2 BT$_n$ subgroup of $X[p^\infty] $ over $C$. Such $G_n$ provides nonzero elements in $\im f^*$ over each fiber.
\end{proof}

Now $Y \ra C$ satisfies 
\begin{enumerate}
\item $Y/C$ has the maximal Higgs field,
\item $\DD(Y[p^\infty]) \cong \cV \otimes \cT$ where $\cV$ is a Dieudonne crystal and $\cT$ is a rank $2^m$ unit root crystal.
\end{enumerate}

The last step is based on Serre-Tate theory. For a given ring $R$ and an ideal $I \subset R$, let $R_0 = R/I$. Write $AS(R)$ for the category of abelian schemes over $R$ and write $BT-\Def(R_0, R)$ for the category of triples $(A_0, G, \a)$ where $A_0$ is an abelian scheme over $R_0$, $G$ is a BT group over $R$ and $\a$ is an isomorphism $\a: G_0=G\otimes R_0 \cong A_0[p^\infty]$.

\begin{thm} \cite[Chapter 5, 1.6]{Messing} \label{Serre-Tate}
Let $R$ be a ring in which the prime $p$ is nilpotent, let $I\subset R$ be a nilpotent ideal and write $R_0=R/I$. Then the functor $AS(R) \ra BT-\Def(R, R_0)$ obtained by sending $A$ to the triplet
\[(A_0=A\otimes R_0, A[p^\infty], \mbox{the natural isomorphism $\a$})\] is an equivalence.
\end{thm}

\begin{thm} \label{the last one}
The family of polarized abelian varieties $Y \ra C$ as above has a unique lifting to $W(k)$ which is a family of abelian fourfolds with maximal Higgs field. 
\end{thm}

\begin{proof}
By \ref{equiv on curve}, the rank 2 Dieudonne crystal $\cV$ corresponds to a height 2 BT group $G$ and the unit root crystal $\cT$ corresponds to a height $2^m$ etale BT group. By the main theorem in \cite{Xia},  the height 2 BT group gives a lifting $\tilde C$ of $C$ to $W(k)$.  The lifting of $G$ gives a lifting of $Y[p^\infty]$. By \ref{Serre-Tate}, $Y \ra C$ lifts to a formal aeblian scheme over $\tilde C$ over $W(k)$.  Again by \cite[Theorem 7.22]{Xia}, the polarization lifts as well and hence the family $Y \ra C$ lifts to a polarized abelian scheme $\tilde Y \ra \tilde C$ over $W(k)$.

Since the special fiber has maximal Higgs field, so is $\tilde Y \ra \tilde C$. 
\end{proof}

\begin{rmk} \label{the Hodge filtration}
From the proof of \ref{the last one}, the weight 1 variation of Hodge structure $R^1\tilde f_*(\O^._{\tilde Y/\tilde C}) \cong \cV_1 \otimes \cdots \otimes \cV_{m+1}$ has the Hodge filtration coming from $\sL \hookrightarrow \cV_1$. 
\end{rmk}

To show the $\tilde C$ is a Mumford curve, we use Theorem 0.5 in \cite{Zuo}. The family in our case is smooth so that there is no unitary part. Maximal Higgs field implies the family reaches the Arakelov bound. Apply the above theorem and we have $\tilde Y \ra \tilde C$ is isogenous to a Mumford curve. In particular, the generic fiber $(\tilde Y \ra \tilde C)\otimes \C$ is a Shimura curve (with a universal family). From \ref{the Hodge filtration}, there can be more than one copy of Mumford curve $Z$ appearing in the decomposition of $\tilde Y$.  Therefore $(\tilde Y \ra \tilde C)\otimes \C$ is a Mumford curve.  

It finishes the proof of the case with maximal Higgs field. 

\section{Without maximal Higgs field} \label{without maximal Higgs field}

In this part, we will prove without maximal Higgs field, $X \ra C$ is a weak Shimura curve over $k$. Note the arguments before \ref{the Higgs field is maximal} still hold for the non-maximal Higgs field case. So $\cE$ can be decomposed to the tensor product $\cV \otimes \cT$ with $\cV$ a rank 2 Dieudonne crystal and $\cT$ a rank $2^m$ unit root crystal.  Equivalently, there exists a family of abelian varieties $Y \ra C$ such that 
\[Y[p^\infty]=G \odot H\]
 where $G$ is a height 2 BT group and $H$ is a height $2^m$ etale BT group. And the family $Y \ra C$ induces a morphism $\phi: C \ra \cal{A}_{2^m, d}$

Let us introduce the categories $\cC$ and $\hat \cC$, following \cite{Schlessinger}. The objects of $\cC$ are the artinian local algebras $R$ such that $R/\frak{m}_R \cong k$. The morphisms in $\cC$ are the homomorphisms of algebras. Then $\hat \cC$ is defined as the category of complete noetherian local algebras $\cal{R}$ such that $\cal{R}/\frak{m}^i_\cal{R}$ is in $\cC$ for all $i$. Again the morphisms are just the homomorphisms of algebras. Notice that $\cC$ is a full subcategory in $\hat \cC$. We consider the $\cal{R}$ in $\hat \cC$ with their $\frak{m}_\cal{R}-$adic topology; for $R$ in $\cC$ this is just the discrete topology.

For any scheme $Z_0$ over $k$, we define a formal deformation functor $\Def_{Z_0}: \cC \ra \mbox{Sets}$, given by 
\[\Def_{Z_0}(R)=\left\{\begin{aligned}
&\mbox{isomorphism classes of pairs ($Z, \psi$), where $Z$ is an scheme}\\
&\mbox{over $\Spec R$ and $\psi$ is an isomorphism $\psi: Z \otimes k \cong Z_0$} 
\end{aligned}\right\} \]

In particular, we can define deformation functor $\Def_{X_0, \l}$ for an abelian scheme $X_0$ and its polarization $\l$. And we can extend these deformation functors to the category $\hat \cC$ by defining $\Def_*(\cal{R})$ as the projective limit of the $\Def_*(\cal{R}/\frak{m}^i_\cal{R})$. 

Note the functor $\Def_{X_0}$ is isomorphic to $\Hom(T_pX_0 \otimes T_p X^t_0, \hat \GG_m)$ and the functor $\Def_{(X_0, \l)}$ is represented by the formal completion of $\cal{A}_{g,d,n}$ at $X_0$ which is a subformal torus of  $\Hom(T_pX_0 \otimes T_p X^t_0, \hat \GG_m)$. 

\begin{prop} \label{image with maximal Higgs field}
Under the above assumption, if  $Y$ is not isotrivial, then the universal family over $\im \phi$ has maximal Higgs field.
\end{prop}

\begin{proof}

Since $Y\ra C$ is not isotrivial, the morphism $\phi$ is not a point. Then the Higgs field of $\im \phi$ is not zero.

For any closed point $c \in C$, let $G_c$ be the restriction of the height 2 BT group $G$ to the point $c$. Since $Y$ is not isotrivial, $\im(\Spec \hat{\cO}_{C,c} \ra \Def_{G_c})$ has dimension $\geq 1$. By \cite{Ill}, the local formal deformation space $\Def_{G_c}$ over $k$ is one-dimensional and isomorphic to $k[[t]]$. We know $\Spec \hat\cO_{C,c} \ra \Def_{G_c}$ is a surjection. 

Let $\l$ be a polarization on $Y_c$. Then $Y_c$ has two deformation spaces  $\Def_{Y_c}$ and $\Def_{Y_c, \l}$. And there exists a natural embedding $\Def_{(Y_c, \l) }\ra \Def_{Y_c}$.

Since $Y_c[p^\infty]\cong G_c \odot H_c$, by Serre-Tate theory(\ref{Serre-Tate}), a deformation of $G_c$ induces deformation of $Y_c$. Thereby we have a morphism between local deformation space $\Def(G_c) \ra \Def(Y_c)$ which is a local embedding. 

Consider \[\xymatrix{
& \Spec \hat \cO_{C,c} \ar[dr] \ar@{->>}[dl]&\\
\Def(G_c) \ar@{-->}[rr]\ar[dr]&&\Def(Y_c, \l) \ar@{^{(}->}[dl]\\
& \Def(Y_c)&.
}\] Since $\im (\Spec \hat \cO_{C,c} \ra \Def(Y_c))\cong \Def(G_c)$ and $\Def_{(Y_c, \l) }\ra \Def_{Y_c}$ is an embedding, we have a morphism $\Def(G_c) \ra \Def(Y_c, \l)$.  Thereby $\im(\Spec \hat \cO_{C,c} \ra \Def(Y_c, \l)) \cong \Def(G_c)\cong k[[t]]$. Thus $\cO_{\im \phi, c} $ maps surjectively to $\Def(G_c)$ which implies the universal family over $\im \hat \phi$ has maximal Higgs field.
\end{proof}

\begin{rmk} \label{the variation}
Now denote $\im \phi$ as $C'$. 
\begin{enumerate}
\item From the proof of \ref{image with maximal Higgs field}, $C'$ has at worst ordinary singularity.
\item We can apply the \ref{the last one} to the normalization of $C'$ and then obtain $Y \ra C'$ is a semistable reduction of a special Mumford curve in $\cA_{2^m, d, n}\otimes \C$. 
\end{enumerate}
\end{rmk}
Then the following proposition justifies the non-maximal Higgs field case.

\begin{prop} \label{lifting of isogeny}
Notations as \ref{main thm}. If $Y\ra C$ is another family of polarized abelian varieties over $C$ such that 
\begin{enumerate}
\item $Y_c $ is ordinary for some closed point $c\in C$, 
\item the image of $C \ra \cA_{2^m,d,n}\otimes k$ induced by $Y$ is a reduction of a special Mumford curve,
\item $f: Y \ra X$ is an isogeny compatible with polarization.
\end{enumerate}
Then the image of $C \ra \cA_{2^m,1,n}$ induced by $X/C$ is an irreducible component of a reduction of a special Mumford curve.
\end{prop}

The proof of \ref{lifting of isogeny} takes the rest of the paper. First let us review the isogeny scheme, referring to \cite{FC}.
\subsection{The isogeny scheme}
Let Isog$_g$ be the moduli stack of isogenies between polarized abelian schemes of relative dimension $g$, so that for an arbitrary base $S$, Isog$_g(S)$ is the category in groupoids whose objects are the isogenies $\phi: A_1 \ra A_2$ over $S$ between polarized abelian schemes $(A_1 \ra S, \l_1)$ and $(A_2 \ra S, \l_2)$ of relative dimension $g$ such that $\phi^* \l_2 = p^e. \l_1$ for some $e\in \mathbb{N}$. Here we do not require $A_i$ to be principally polarized.

Assigning $\phi$ to its source(resp. target) defines a morphism pr$_1:$Isog$_g\ra\cA_{g,d}$(resp. pr$_2:$Isog$_g \ra \cA_{g,d'}$). Bounding the degree of the isogeny gives a substack of Isog$_g$. We write Isog$(p^l)$ for the stack of $p-$isogenies of degree less than or equal to $p^l$ which is of finite type over $\cA_{g,d}$.

As a variant, we can take level structure into account. Choose an integer $n>3$ and consider $\cA_{g,d,n}$. Making isogenies compatible with level structures, we obtain a scheme Isog$(p^l)$ over $\cA_{g,d,n} \times \cA_{g,d',n}$. To keep notations easy, we omit the subscripts $n$.

Let Isog$^o(p^l)_g$ (resp. $\cA^o_{g,d}$) denote the ordinary locus of Isog$_g(p^l)$ (resp. $\cA_{g,d}$) over any base. 

We need the following result to prove \ref{lifting of isogeny}. To ease the notations, we do not write the bound of the degree $p^l$. 
\begin{prop} \label{finite and flat}
The projections pr$_i:\text{Isog}^o_g\ra \cA^o_{g,1}$ are finite and flat.
\end{prop}

\begin{proof}
This proposition is essentially similar to the Proposition 4.1(i) in the Chapter VII of \cite{FC}.  Here we prove it directly.

It is known that the morphism $\mbox{Isog}_g \ra \cA_{g,1}$ is proper. Let $Y$ be any ordinary abelian variety over any algebraically closed field.  Note the degree of the isogeny is bounded and there is only finitely many subgroup in $Y[p^n]\cong Y^{\et}[p^n] \times Y^{\text{mult}}[p^n]\cong \m_{p^n}\times \Z/p^n$. Therefore the fiber of projection $\mbox{Isog}_g \ra \cA_{g,1}$ over $X$ is finite which implies that pr$:\text{Isog}^o_g\ra \cA^o_{g,1}$ is quasi-finite. So pr is a finite morphism.

Let $\eta$ (resp. $\xi$) be a closed point in $\cA_{g,1}$(resp. $\mbox{Isog}_g$) with stalk $A_\eta$(resp. $B_\xi$) and assume pr$(\xi)=\eta$. Since both $A_\eta$ and $B_\xi$ are local rings, by local criterion of flatness, it suffices to show the image of $A_\eta$ in $B_\xi$ contains no zerodivisor. Then we can assume $A_\eta$ and $B_\xi$ are complete with respect to the maximal ideals. 

Let $t$ be any element in the image of $A_\eta$ and $\frak{q}$ be a minimal point in the zero set $V(t)\subset \Spec B_\xi$. Then $\mbox{ht}\frak{q} \leq 1$. Let $\frak{p}=\frak{q}\cap A_\eta$.  

On the one hand, $t$ is a zerodivisor if and only if $\mbox{ht\,}\frak{q}=0$. 

On the other hand, for the principally polarized abelian variety $X$ defined over $\Spec (A_\eta)_ \frak{p}$ and an isogeny $g: X\otimes k \ra Y$ over the special fiber of $\Spec (A_\eta)_ \frak{p}$. Let $K = \ker g$ and thus $K\subset X_k[p^e]\cong X_k[p^e]^{\mult} \times X_k[p^e]^{\et}$. Then correspondingly, $K=K^\mult \times K^\et$. Since $A_\eta$ is complete, the splitting $X_k[p^e]\cong X_k[p^e]^{\mult} \times X_k[p^e]^{\et}$ can be extended to $\Spec (A_\eta)_ \frak{p}$. Then $K^\mult$ and $K^\et$ can also be lifted to $\Spec (A_\eta)_ \frak{p}$. Thereby we have a deformation of the BT group $Y[p^\infty]$. By Serre-Tate theory( \ref{Serre-Tate}), the abelian variety $Y$ can be lifted to a formal abelian scheme over $\mathrm{Spf} (A_\eta)_ \frak{p}$.  Since $Y$ is principally polarized, any polarization on $Y$ can be lifted due to the uniqueness of the lifting of $Y$. Thus the isogeny $\p$ can be lifted.

Therefore the image of $\Spec (B_\xi)_\frak{q} \ra \Spec (A_\eta)_\frak{p}$ is not just the special fiber, i.e. $\mbox{ht}(\frak{q})>0$. So pr is flat over the point $Y$.
\end{proof}

\begin{rmk} \label{general degree}
In the proof of \ref{finite and flat}, the principal polarization is only used in the last step, i.e. to lift the polarization. Therefore, if $X$ is principally polarized and the polarization $\l_Y$ satisfies
\[\xymatrix{
Y \ar[r] \ar[d]^{\l_Y} & X \ar[d]^{\l_X} \\
Y^t  & X^t \ar[l]
}\] is commutative, then the proof still works.  So we have $\mathrm{Isog}^o_g \ra \cA^o_{g,\deg \l_Y}$ is also flat at the point corresponding to $Y$. 
\end{rmk}

\subsection{Proof of \ref{lifting of isogeny}}
In our case, we mainly focus on the Isog$_{2^m}$ associated to $\cA_{{2^m},1}$ and $\cA_{{2^m},d}$ where $d=2|K|$. 
\[\xymatrix{
&\text{Isog}_{2^m} \ar[dr]^{pr_2} \ar[dl]_{pr_1} &\\
\cA_{{2^m},1}&&\cA_{{2^m},d}
}
\]

Let us denote the lifting of the image of $C \ra \cA_{2^m,d,n}$ as $\tilde C$ and $\cO_{\tilde C, \eta}$ be the local ring of $\tilde C$ at $\eta$. The isogeny $Y_\eta \ra X_\eta$ gives a point $\xi \in $Isog$^o_{2^m}$ such that pr$_2(\xi)=\eta$. By \ref{general degree}, pr$_2$ is flat over $\xi$. And from the openness of the flatness, there exists a subscheme $\Spec R \subset $Isog$_{2^m}$ of dimension 1 with special fiber $\xi$ such that pr$_2(\Spec R)=\Spec \cO_{\tilde C, \eta}$. Therefore the generic kernel $K_\eta$ of the isogeny can be lifted to $\Spec R$. Denote the lifting as $\tilde K_\eta$. Let $\tilde C'$ be the normalization of $\tilde C$ in the fractional field of $R$. Since $\Spec R \ra \Spec \cO_{\tilde C, \eta}$ is a finite morphism, we have the following commutative diagram: 
\[\xymatrix{
\Spec R \ar[d] \ar[r] & \tilde C' \ar[d]\\
\Spec \cO_{\tilde C, \eta} \ar[r] & \tilde C .\\
}\]
And the closure of $\Spec R$ in $\tilde C'$ is the whole $\tilde C'$. Then take the closure of $\tilde K_\eta$ over $\tilde C'$ and hence the quotient 
\[\tilde Y' \ra \tilde Y'/(\tilde K_\eta)^-=: \tilde X' \] 
lifts the isogeny $Y \ra X$ to $\tilde C'$ over $W(k)$. For simplicity, let us still denote the image of $\tilde C'$ in $\mathrm{Isog}_{2^m}$ as $\tilde C'$.

By \ref{the variation},  the family $\tilde Y \ra \tilde C$ over $W(k)$ whose generic fiber is a special Mumford curve in $\cA_{2^m, d, n}\otimes \C$. 
Note $\mathrm{pr}_2(\tilde C') =\tilde C$ and $\mathrm{pr}_1(\tilde C')$ is connected to $\tilde C$ via Hecke correspondence. 
Then $\mathrm{pr}_1(\tilde C')$ also has a special Mumford curve as the generic fiber. Note $\tilde X' \ra \mathrm{pr}_1(\tilde C')$ is a lifting of the image of $C \ra \cA_{2^m, 1, n}$ induced by $X$. Since the reduction of $\tilde C'$ may be reducible, we can at best have that the image of $C \ra \cA_{{2^m},1,n}$ is an irreducible component of the reduction. So $X \ra C$ is a weak Shimura curve over $k$.

Now we complete the proof of \ref{main thm}.

\begin{appendix}

\section{crystal model} \label{crystal model}

Assume we have a $F$-isocrystal $\cV$ over $C$. Then it is in prior a crystal over $C$. Locally, it corresponds to a module with connection $(M. \nabla)$ over $\tilde C$ such that $M \otimes B(k)$ admits a $\s-$linear morphism $F$. The point is to descend the Frobenius $F$.

Though $F$ may not descend to $M$, we can consider $M'=\sum_n F^{(n)}(M)$. Then $M \subset M' \subset M\otimes B(k_0)$. One can mimic the the proof of ( \cite{Katz}, Theorem 2.6.1) to show $M'$ is finitely generated provided that all slopes are nonnegative. 

Since $\tilde C$ is regular of dimension 2, taking the double dual $M'^{\vee\vee}$ gives a locally free sheaf over $\tilde C$. 

If $\cV$ is an isocrystal with slopes between 0 and 1, then we can further choose a morphism $V: \cV\ra \cV^\s$ such that $V\circ F= F\circ V =p.$ Hence such isocrystal has a model of Dieudonne crystal.

\end{appendix}

\bibliographystyle{amsplain}
\bibliography{mybib}{}

\end{document}